\documentclass[11pt,reqno]{amsart}
\usepackage{amsmath, amsthm, amssymb, amsbsy}
\usepackage[letterpaper=true,colorlinks=true,urlcolor=blue,linkcolor=blue,citecolor=blue]{hyperref}

\newtheorem{lemma}{Lemma}[section]
\newtheorem{theorem}{Theorem}[section]
\newtheorem{proposition}[lemma]{Proposition}

\theoremstyle{definition}

\numberwithin{equation}{section}

\def\le{\leqslant}

\def\pmod #1{\ ({\rm mod}\ #1)}

\def \N {{\mathbb N}}

\def \Z {{\mathbb Z}}

\def\f{\frac}

\begin{document}
\title[restricted sumsets]
{On restricted sumsets over a field}
\author{Lilu ZHAO}
\email{zhaolilu@gmail.com}
\address{School of Mathematics, Hefei University of Technology, Hefei,
People's Republic of China}

\begin{abstract}
We consider restricted sumsets over field $F$.
Let\begin{align*}C=\{a_1+\cdots+a_n:a_1\in A_1,\ldots,a_n\in A_n,
a_i-a_j\notin S_{ij}\ \text{if}\ i\not=j\},\end{align*} where
$S_{ij}(1\leqslant i\not=j\leqslant n)$ are finite subsets of $F$
with cardinality $m$, and $A_1,\ldots, A_n$ are finite nonempty
subsets of $F$ with $|A_1|=\cdots=|A_n|=k$. Let $p(F)$ be the
additive order of the identity of $F$. It is proved that
$|C|\geqslant \min\{p(F),\ \ n(k-1)-mn(n-1)+1\}$ if $p(F)>mn$.
This conclusion refines the result of Hou and Sun \cite{HS}.
\end{abstract}

\maketitle


{\let\thefootnote\relax}

{\let\thefootnote\relax}

\section{Introduction}

Let $F$ be a field. Denote by $p(F)$ the additive order of the
identity of $F$. It is well-known that $p(F)$ is either infinite
or a prime. For a finite set $A$, we use $|A|$ to denote the
cardinality of $A$.

 Suppose that $A_1,\ldots, A_n$ are finite nonempty subsets of $F$ with $|A_j|=k_j$ for $1\leqslant j\leqslant
 n$. The Cauchy-Davenport Theorem asserts that
\begin{align*}|\{a_1+\cdots+a_n:\ a_1\in A_1,\ldots,a_n\in A_n\}|\geqslant \min\{p(F),k_1+\cdots+k_n-n+1\}.\end{align*}
Let $A$ be a finite subset of $F$. We
define\begin{align*}n^{\wedge}A=\{a_1+\cdots+a_n:a_1,\ldots,a_n\in
A,\ a_1,\ldots,a_n \ \textrm{are distinct}\}.\end{align*} P.
Erd\H{o}s and H. Heilbronn \cite{EH} conjectured that
$$|2^{\wedge}A|\geqslant \min\{p(F),\ 2|A|-3\}.$$ This conjecture
was solved by Dias da Silva and Hamidoune \cite{DH}, who
established
\begin{align*}|n^{\wedge}A|\geqslant \min\{p(F),\ n|A|-n^2+1\}.\end{align*}
In 1995-1996, Alon, Nathanson and Ruzsa \cite{ANR1,ANR2} developed
the polynomial method to show if $0<k_1<k_2<\cdots<k_n$, then
\begin{align*}&|\{a_1+\cdots+a_n:a_i\in A_i,\ a_1,\ldots,a_n \ \textrm{are distinct}\}|
\geqslant \min\{p(F),\sum_{j=1}^{n}(k_j-j)+1\}.\end{align*}
Various restricted sumsets of $A_1,\ldots, A_n$ were
 investigated in \cite{HS}, \cite{LiuSun}, \cite{PS}, \cite{Sun},
 \cite{Sun2} and \cite{SunY}.
In particular, Hou and Sun \cite{HS} considered the following
sumset
\begin{align}\label{defC}C=\{a_1+\cdots+a_n:a_1\in
A_1,\ldots,a_n\in A_n, a_i-a_j\notin S_{ij}\ \text{if}\
i\not=j\},\end{align} where $S_{ij}(1\leqslant i\not=j\leqslant
n)$ are finite subsets of $F$. Hou and Sun \cite{HS} established
the following result.
\begin{theorem}[Hou-Sun]
\label{t1} Let $C$ be given by (\ref{defC}) with
$|S_{ij}|=m(1\leqslant i\not=j\leqslant n)$. If
$|A_1|=\cdots=|A_n|=k$ and $p(F)>\max\{n(k-1)-mn(n-1),mn\}$, then
$$|C|\geqslant n(k-1)-mn(n-1)+1.$$
\end{theorem}

The first result of this paper is to refine Theorem \ref{t1}.
\begin{theorem}
\label{t2} Let $S_{ij}(1\leqslant i\not=j\leqslant n)$ be finite
subsets of $F$ with cardinality $m$, and let $C$ be defined in
(\ref{defC}). Suppose that $|A_j|\in\{k,k+1\}$ for $1\leqslant
j\leqslant n$. If $p(F)>\max\{mn,
\sum_{j=1}^n(|A_j|-1)-mn(n-1)\}$, then
\begin{equation}\label{1.2}|C|\geqslant
\sum_{j=1}^n(|A_j|-1)-mn(n-1)+1.\end{equation}
\end{theorem}

It was pointed out by Hou and Sun \cite{HS} that when
$p(F)\leqslant n(k-1)-mn(n-1)$, Theorem \ref{t1} only implies
\begin{align}\label{C}|C|\geqslant n\Big\lfloor \frac{p(F)-1}{n}\Big\rfloor+1,\end{align}
where $\lfloor\alpha\rfloor$ denotes the greatest integer not
exceeding real number $\alpha$. Now we can deduce the following
result from Theorem \ref{t2}.

\begin{theorem}
\label{t3} Suppose that $|S_{ij}|=m(1\leqslant i\not=j\leqslant
n)$, $|A_1|=\cdots=|A_n|=k$ and $p(F)>mn$. Let $C$ be defined in
(\ref{defC}). We have
$$|C|\geqslant \min\big\{p(F),\ \ n(k-1)-mn(n-1)+1\big\}.$$
\end{theorem}

When $p(F)\leqslant n(k-1)-mn(n-1)$, Theorem \ref{t3} implies
$|C|\geqslant p(F)$, which improves upon the inequality (\ref{C}).
Dias da Silva and Hamidoune \cite{DH} showed that if
$|A|>\sqrt{4p-7}$ then any element of $\Z/p\Z$ is a sum of
$\lfloor\frac{|A|}{2}\rfloor$ distinct elements of $A$. We extend
this result with the help of Theorem \ref{t3}.
\begin{theorem}
\label{t4} Let $p$ be a prime and $S$ be a subset of $F=\Z/p\Z$
with $|S|=m$. Let $A$ be a subset of $F$ with $|A|\geqslant
\sqrt{4mp+4m(m-3)+2}-m+1$. Then any element of $F$ can be written
in the form $a_1+\dots+a_n$ with
$n=\lfloor\frac{|A|-1+m}{2m}\rfloor$ and $a_i-a_j\not\in S$ if
$1\leqslant i\neq j\leqslant n$.
\end{theorem}

Section 2 is devoted to some preparations. The proofs of Theorems
\ref{t2}-\ref{t3} will be given in Section 3. Finally, we deduce
Theorem \ref{t4} from Theorem \ref{t3}.

 \vskip5mm
\section{Preliminaries}

For a polynomial $g(x_1,\ldots,x_n)$ over a field $F$, by
$[x_1^{k_1}\cdots x_n^{k_n}]g(x_1,\ldots,x_n)$ we mean the
coefficient of the monomial $x_1^{k_1}\cdots x_n^{k_n}$ in
$g(x_1,\ldots,x_n)$. One has the following tool of the polynomial
method. The reader may refer to the book of Tao and Vu [17, pp.
329-345] for the explanation of the polynomial method.

\medskip
\begin{lemma}[\cite{Al},\cite{ANR2}]\label{L0} Let $A_1,\cdots,A_n$ be
non-empty finite subsets of a field $F$, and let
$P(x_1,\cdots,x_n)\in F[x_1,\cdots,x_n]\backslash \{0\}$. Suppose
that $\deg P\leqslant \sum_{j=1}^n(|A_j|-1)$. If
$$[x_1^{|A_1|-1}\cdots x_n^{|A_n|-1}]P(x_1,\cdots,x_n)
(x_1+\cdots+x_n)^{\sum_{j=1}^n(|A_j|-1)-\deg P}\neq 0,$$ then
$$|\{a_1+\cdots+a_n:a_1\in A_1,\ldots,a_n\in
A_n, P(a_1,\ldots,a_n)\not= 0\}|\geqslant
\sum_{j=1}^n(|A_j|-1)-\deg P+1.$$
\end{lemma}

\medskip

For nonnegative integers $a_0,a_1,\ldots, a_n$, F. J. Dyson
\cite{Dy} in 1962 conjectured that the constant term of
$\prod_{0\leqslant i\not=j\leqslant n}(1-\f{x_i}{x_j})^{a_j}$ is
$\frac{(a_0+\cdots+a_n)!}{a_0!\cdots a_n!}$. This conjecture was
proved independently by Gunson \cite{Gu} and by Wilson \cite{W}.
I. J. Good \cite{Go} in 1970 used the Lagrange interpolation
formula to provide a short proof. D. Zeilberger \cite{Z} gave a
combinatorial proof of Dyson's conjecture in the following
equivalent form
$$[x_0^{a_0}\cdots x_n^{a_n}]\prod_{0\leqslant i<j\leqslant n}
(x_i-x_j)^{a_i+a_j}=(-1)^{\sum_{j=0}^{n}(j+1)a_j}\frac{(a_0+\cdots+a_n)!}{a_0!\cdots
a_n!}.$$ Aomoto \cite{Ao} proved that the constant term of
$$\prod_{l=1}^{n}(1-\frac{x_l}{x_0})^{a+\chi (l\leqslant s)}
(1-\frac{x_0}{x_l})^b\prod_{1\leqslant i\not=j\leqslant
n}(1-\frac{x_i}{x_j})^{m}$$ is
$$\prod_{l=0}^{n-1}\frac{\big(a+b+ml+\chi(l\geqslant
n-s)\big)!(ml+m)!}{\big(a+ml+\chi(l\geqslant
n-s)\big)!(ml+b)!m!},$$ where $\chi(l\geqslant t)=1$ if
$l\geqslant t$, and $\chi(l\geqslant t)=0$ otherwise.

We can deduce the following result from Aomoto's identity.
\begin{proposition}\label{L1}
Let $m\in \N$ and $k,\ n\in \Z^+$. Suppose that $k_j\in \{k,k+1\}$
for $1\leqslant j\leqslant n$. If $k>m(n-1)$, then we have
\begin{align*}&[x_1^{k_1-1}\cdots x_n^{k_n-1}]\prod_{1\leqslant i\not=j\leqslant
n}(x_i-x_j)^{m}(x_1+\cdots+x_n)^{\sum_{j=1}^{n}(k_j-1)-mn(n-1)}
\\=&\Big(\prod_{j=0}^{s-1}\frac{1}{k-jm}\Big)\frac{(\sum_{j=1}^{n}k_j-mn^2+mn-n)!}{(m!)^n}
\prod_{j=1}^{n}\frac{(jm)!}{(k-1-jm+m)!},
\end{align*}
where $s=|\{1\leqslant j\leqslant n: k_j=k+1\}|$.
\end{proposition}

In order to prove Proposition \ref{L1}, we also need the following
result (see Lemma 2.1 and Corollary A.1 in \cite{GLXZ}).

\begin{lemma}[Gessel-Lv-Xin-Zhou]\label{L2}
Let $a_0,a_1,\ldots, a_n$ be nonnegative integers, and let
$L(x_1,\ldots,x_n)$ be a Laurent polynomial independent of $a_0$.
Then the constant term of
$$\prod_{l=1}^{n}(1-\frac{x_l}{x_0})^{a_0}(1-\frac{x_0}{x_l})^{a_l}\,L(x_1,\ldots,x_n)$$
is a polynomial in $a_0$ for fixed $a_1,\ldots, a_n$ and the
leading coefficient of such polynomial in $a_0$ coincides with the
constant term of
$$\f{1}{(a_1+\cdots+a_n)!}(x_1+\cdots+x_n)^{a_1+\cdots+a_n}
\prod_{l=1}^{n}x_l^{-a_l}\,L(x_1,\ldots,x_n).$$
\end{lemma}
\medskip

\begin{proof}[Proof of Proposition \ref{L1}] Without loss of generality,
we assume that $k_1=\cdots=k_s=k+1$ and $k_{s+1}=\cdots=k_n=k$ for
some $s$. For nonnegative integers $a$, $b$, we define
$f(a;b,s,m)$ to be the constant term of
$\mathcal{F}(x_0,x_1,\ldots,x_n)$, where
$$\mathcal{F}(x_0,x_1,\ldots,x_n)=
\prod_{l=1}^{n}(1-\f{x_l}{x_0})^{a}(1-\f{x_0}{x_l})^{b+\chi
(l\leqslant s)}\prod_{1\leqslant i\not=j\leqslant
n}(1-\f{x_i}{x_j})^{m}.$$ On substituting $x_l=y_l^{-1}$ for
$0\leqslant l\leqslant n$, we get
$$\mathcal{F}(y_0^{-1},y_1^{-1},\ldots,y_n^{-1})=\mathcal{G}(y_0,y_1,\ldots,y_n),$$ where
$$\mathcal{G}(y_0,y_1,\ldots,y_n)=
\prod_{l=1}^{n}(1-\f{y_0}{y_l})^{a}(1-\f{y_l}{y_0})^{b+\chi
(l\leqslant s)}\prod_{1\leqslant i\not=j\leqslant
n}(1-\f{y_j}{y_i})^{m}.$$ We observe
\begin{align}\label{G}\mathcal{G}(y_0,y_1,\ldots,y_n)=
\prod_{l=1}^{n}(1-\f{y_l}{y_0})^{b+\chi (l\leqslant
s)}(1-\f{y_0}{y_l})^{a}\prod_{1\leqslant i\not=j\leqslant
n}(1-\f{y_i}{y_j})^{m}.\end{align} One can see that $f(a;b,s,m)$
is equal to the constant term of
$\mathcal{G}(y_0,y_1,\ldots,y_n)$. By (\ref{G}) and Aomoto's
identity, we conclude
\begin{equation}\label{eq22}f(a;b,s,m)=\prod_{l=0}^{n-1}\f{\big(a+b+ml+\chi(l\geqslant
n-s)\big)!(ml+m)!}{(ml+a)!\big(b+ml+\chi(l\geqslant n-s)\big)!m!}.
\end{equation} Then applying Lemma \ref{L2} with $a_0=a$,
$a_l=b+\chi(l\leqslant s)$ for $1\leqslant l\leqslant n$ and
$L(x_1,\ldots,x_n)=\prod_{1\leqslant i\not=j\leqslant
n}(1-\f{x_i}{x_j})^{m}$, the leading coefficient of $f(a;b,s,m)$
in $a_0$ coincides with the constant term of
\begin{align*}\f{1}{(a_1+\cdots+a_n)!}(x_1+\cdots+x_n)^{a_1+\cdots+a_n}\prod_{l=1}^{n}x_l^{-a_l}\,\prod_{1\leqslant
i\not=j\leqslant n}(1-\f{x_i}{x_j})^{m}.\end{align*} Note that
\begin{align*}&(x_1+\cdots+x_n)^{a_1+\cdots+a_n}\prod_{l=1}^{n}x_l^{-a_l}\,\prod_{1\leqslant
i\not=j\leqslant
n}(1-\f{x_i}{x_j})^{m}\\=&(x_1+\cdots+x_n)^{a_1+\cdots+a_n}\prod_{l=1}^{n}x_l^{-(a_l+(n-1)m)}\,\prod_{1\leqslant
i\not=j\leqslant n}(x_i-x_j)^{m}.\end{align*} The constant term of
$(x_1+\cdots+x_n)^{a_1+\cdots+a_n}\prod_{l=1}^{n}x_l^{-a_l}\,\prod_{1\leqslant
i\not=j\leqslant n}(1-\f{x_i}{x_j})^{m}$ is
$$\Big[\prod_{l=1}^{n}x_l^{a_l+(n-1)m}\Big]\prod_{1\leqslant
i\not=j\leqslant
n}(x_i-x_j)^{m}\,(x_1+\cdots+x_n)^{a_1+\cdots+a_n}.$$ By
(\ref{eq22}), the leading coefficient of $f(a;b,s,m)$ in $a$ is
$$\prod_{l=0}^{n-1}\f{(ml+m)!}{\big(b+ml+\chi(l\geqslant n-s)\big)!m!}.$$
Now we can conclude that
\begin{equation}\label{fin}
\begin{split}[\prod_{l=1}^{n}x_l^{a_l+(n-1)m}]&\prod_{1\leqslant i\not=j\leqslant
n}(x_i-x_j)^{m}\,(x_1+\cdots+x_n)^{a_1+\cdots+a_n}\\=\f{(a_1+\cdots+a_n)!}{(m!)^n}&\prod_{l=0}^{n-1}\f{(ml+m)!}{\big(b+ml+\chi(l\geqslant
n-s)\big)!}. \end{split}\end{equation} On substituting
$b=k-(n-1)m-1$, we get the desired result from (\ref{fin}). The
proof is completed.
\end{proof}

\medskip

\section{Proofs of Theorems \ref{t2}-\ref{t4}}

\medskip

\begin{proof}[Proof of Theorem \ref{t2}] Since (\ref{1.2}) holds trivially if $\sum_{j=1}^n(|A_j|-1)-mn(n-1)<0$. Below we assume that
$\sum_{j=1}^n(|A_j|-1)\geqslant mn(n-1)$. Define
$$P(x_1,\cdots,x_n)=\prod_{1\leqslant i\not=j\leqslant n}\prod_{s\in
S_{ij}}(x_i-x_j-s).$$ We observe that $\deg(P)=mn(n-1)$, and
\begin{align*}C=\{a_1+\cdots+a_n:a_1\in A_1,\ldots,a_n\in
A_n, P(a_1,\ldots,a_n)\neq 0\}.
\end{align*}
Our objective is to
prove\begin{equation}\label{not0}[x_1^{|A_1|-1}\cdots
x_n^{|A_n|-1}]P(x_1,\cdots,x_n)(x_1+\cdots+x_n)^{\sum_{j=1}^n(|A_j|-1)-\deg
P} \not=0.
\end{equation}Then the desired conclusion follows from Lemma
\ref{L0}. Note that
\begin{align*}&[x_1^{|A_1|-1}\cdots
x_n^{|A_n|-1}]P(x_1,\cdots,x_n)(x_1+\cdots+x_n)^{\sum_{j=1}^n(|A_j|-1)-\deg
P}
\\=&[x_1^{|A_1|-1}\cdots x_n^{|A_n|-1}]\prod_{1\leqslant i\not=j\leqslant
n}(x_i-x_j)^{m}(x_1+\cdots+x_n)^{\sum_{j=1}^{n}(|A_j|-1)-mn(n-1)}.
\end{align*}Let $s = \{1 \le j \le n : |A_j| = k + 1\}$. By Proposition \ref{L1},
\begin{align*}&[x_1^{|A_1|-1}\cdots
x_n^{|A_n|-1}]P(x_1,\cdots,x_n)(x_1+\cdots+x_n)^{\sum_{j=1}^n(|A_j|-1)-\deg
P}
\\=&\Big(\prod_{j=0}^{s-1}\frac{1}{k-jm}\Big)\frac{(\sum_{j=1}^{n}|A_j|-mn^2+mn-n)!}{(m!)^n}
\prod_{j=1}^{n}\frac{(jm)!}{(k-1-jm+m)!}.
\end{align*}
On recalling the condition $p(F)>\max\{mn,
\sum_{j=1}^n(|A_j|-1)-mn(n-1)\}$, we have established
(\ref{not0}). The proof of Theorem \ref{t2} is completed.
\end{proof}

\medskip
\begin{proof}[Proof of Theorem \ref{t3}]
In view of Theorem \ref{t1}, we only need to consider the case
$p(F)\leqslant n(k-1)-mn(n-1)$. Set
$k'=\lfloor\frac{p(F)-1}{n}\rfloor+m(n-1)+1$. Then there exist
subsets $A_j'\subseteq A_j$ for $1\leqslant j\leqslant n$ such
that
\begin{align*}\sum_{j=1}^{n}(|A_j'|-1)-mn(n-1)=p(F)-1\ \textrm{ and }\ |A_j'|\in
\{k',k'+1\}.\end{align*} We easily see that
\begin{align*}C\supseteq\{a_1+\cdots+a_n:a_1\in A_1',\ldots,a_n\in A_n', a_i-a_j\not\in
S_{ij}\ \textrm{if}\ i\not=j\}.
\end{align*}
Then we apply Theorem \ref{t2} to deduce that
\begin{align*}|C|\geqslant&\ |\{a_1+\cdots+a_n:a_1\in A_1',\ldots,a_n\in A_n', a_i-a_j\not\in
S_{ij}\ \textrm{if}\ i\not=j\}| \\  \geqslant&\
\sum_{j=1}^n(|A_j'|-1)-mn(n-1)+1 \\ =&\ p(F)=\min\big\{p(F), \
n(k-1)-mn(n-1)+1\big\}.
\end{align*}This completes the proof.
\end{proof}
\medskip
\begin{proof}[Proof of Theorem \ref{t4}] Note that since
$n=\lfloor\f{|A|-1+m}{2m}\rfloor$, we have $|A|-1+m\geqslant 2mn$.
Since $A$ is a subset of $F=\Z/p\Z$, we deduce that $p\geqslant
|A|>2mn-m\geqslant mn$. The conclusion of Theorem \ref{t4} is
equivalent to
\begin{equation*}\{a_1+\cdots+a_n:a_j\in A, a_i-a_j\not\in S
\ \text{if}\ i\not=j \}=\Z/p\Z.
\end{equation*}By Theorem \ref{t3}, it suffices to prove
\begin{align*}p\leqslant n(|A|-1)-mn(n-1)+1.\end{align*}

Let $r=2m\big(\f{|A|-1+m}{2m}-n\big)$. We can see that
$|A|-1+m=2mn+r$ and $0\leqslant r\leqslant 2m-1$. Then we obtain
\begin{align}\label{cong}(|A|-1+m)^2=(2mn+r)^2=r^2+4m(mn^2+nr)\equiv r^2\pmod{4m}.\end{align}
In view of the condition $|A|\geqslant \sqrt{4mp+4m(m-3)+2}-m+1$,
we have
\begin{align*}(|A|-1+m)^2\geqslant
4mp+4m(m-3)+2=&4m(p-1)+(2m-1)^2-(4m-1) .\end{align*} Therefore,
\begin{align}\label{lower}(|A|-1+m)^2\geqslant
4m(p-1)+r^2-(4m-1).\end{align} It follows from (\ref{cong}) and
(\ref{lower}) that $(|A|-1+m)^2\geqslant 4m(p-1)+r^2$. This
implies
$$n(|A|-1)-mn(n-1)+1=\f{(|A|-1+m)^2-r^2}{4m}+1\geqslant p.$$
We complete the proof of Theorem \ref{t4}.
\end{proof}
\medskip

{\bf Acknowledgement}. The author would like to thank referees for
the comments and suggestions.

\vskip4mm
\end{document}